\documentclass[12pt]{amsart}
\usepackage{amsmath}
\usepackage{amstext}
\usepackage{amssymb}
\usepackage{amsthm}
\swapnumbers
\theoremstyle{plain}%default
\newtheorem{thm}{Theorem}[section]
\newtheorem{lem}[thm]{Lemma}
\newtheorem{prop}[thm]{Proposition}
\newtheorem{cor}[thm]{Corollary}

\theoremstyle{definition}
\newtheorem{defi}[thm]{Definition}

\newtheorem{ex}[thm]{Example}

\newtheorem{ntn}[thm]{Notation}

\theoremstyle{remark}

\newtheorem*{notes}{Notes}
\newtheorem{rmk}[thm]{Remark}
\newtheorem{disc}[thm]{Discussion}

\DeclareMathOperator{\height}{ht} 
 
\DeclareMathOperator{\ext}{exten}

 \DeclareMathOperator{\HSL}{HSL}
\DeclareMathOperator{\Spec}{Spec} 
 \DeclareMathOperator{\ann}{ann}
\DeclareMathOperator{\grann}{gr-ann} 
\def\Z{\mathbb Z}
\def\N{\mathbb N}

\def\F{\mathbb F}

\def\fa{{\mathfrak{a}}}

\def\fb{{\mathfrak{b}}}
\def\fB{{\mathfrak{B}}}
\def\fc{{\mathfrak{c}}}

\def\fd{{\mathfrak{d}}}

\def\fm{{\mathfrak{m}}}

\def\fp{{\mathfrak{p}}}

\def\nn{\relax\ifmmode{\mathbb N_{0}}\else$\mathbb N_{0}$\fi}
\def\lra{\longrightarrow}

\begin{document}

\title[Graded annihilators and tight closure test ideals]{Graded annihilators and tight closure test ideals}
\author{RODNEY Y. SHARP}
\address{Department of Pure Mathematics,
University of Sheffield, Hicks Building, Sheffield S3 7RH, United
Kingdom} \email{R.Y.Sharp@sheffield.ac.uk}


\date{\today}


\begin{abstract}
Let $R$ be a commutative Noetherian local ring of prime
characteristic $p$, with maximal ideal $\fm$. The main purposes of
this paper are to show that if the injective envelope $E$ of $R/\fm$
has a structure as an $x$-torsion-free left module over the
Frobenius skew polynomial ring over $R$ (in the indeterminate $x$),
then $R$ has a tight closure test element (for modules) and is
$F$-pure, and to relate the test ideal of $R$ to the smallest
`$E$-special' ideal of $R$ of positive height.

A byproduct is an analogue of a result of Janet Cowden Vassilev: she
showed, in the case where $R$ is an $F$-pure homomorphic image of an
$F$-finite regular local ring, that there exists a strictly
ascending chain $0 = \tau_0 \subset \tau_1 \subset \cdots \subset
\tau_t = R$ of radical ideals of $R$ such that, for each $i = 0,
\ldots, t-1$, the reduced local ring $R/\tau_i$ is $F$-pure and its
test ideal (has positive height and) is exactly $\tau_{i+1}/\tau_i$.
This paper presents an analogous result in the case where $R$ is
complete (but not necessarily $F$-finite) and $E$ has a structure as
an $x$-torsion-free left module over the Frobenius skew polynomial
ring. Whereas Cowden Vassilev's results were based on R. Fedder's
criterion for $F$-purity, the arguments in this paper are based on
the author's work on graded annihilators of left modules over the
Frobenius skew polynomial ring.
\end{abstract}

\maketitle

\setcounter{section}{-1}
\section{\sc Introduction}
\label{intro}

This paper presents new information about test elements in tight
closure theory in commutative algebra.

Throughout the paper, $R$ will denote a commutative Noetherian ring
of prime characteristic $p$. We shall always denote by $f:R\lra R$
the Frobenius homomorphism, for which $f(r) = r^p$ for all $r \in
R$. This paper will make use of the author's work in \cite{ga} on
left modules over the {\em Frobenius skew polynomial ring over
$R$\/}, that is, the skew polynomial ring $R[x,f]$ associated to $R$
and $f$ in the indeterminate $x$ over $R$. Recall that $R[x,f]$ is,
as a left $R$-module, freely generated by $(x^i)_{i \in \nn}$ (I use
$\N$ and $\nn$ to denote the set of positive integers and the set of
non-negative integers, respectively),
 and so consists
 of all polynomials $\sum_{i = 0}^n r_i x^i$, where  $n \in \nn$
 and  $r_0,\ldots,r_n \in R$; however, its multiplication is subject to the
 rule
 $
  xr = f(r)x = r^px$ for all $r \in R.$
Note that $R[x,f]$ can be considered as a positively-graded ring
$R[x,f] = \bigoplus_{n=0}^{\infty} R[x,f]_n$, with $R[x,f]_n = Rx^n$
for all $n \in \nn$.

If, for $n \in \N$, we endow $Rx^n$ with its natural structure as an
$(R,R)$-bimodule (inherited from its being a graded component of
$R[x,f]$), then $Rx^n$ is isomorphic (as $(R,R)$-bimodule) to $R$
viewed as a left $R$-module in the natural way and as a right
$R$-module via $f^n$, the $n$th iterate of the Frobenius ring
homomorphism. With this observation, we can formulate the
definitions of $F$-purity and of the tight closure of the zero
submodule in an $R$-module $M$ in terms of the left $R[x,f]$-module
$R[x,f]\otimes_RM$, as follows.

First of all, we can write that $R$ is {\em $F$-pure\/} precisely
when, for each $R$-module $N$, the map $\psi_N : N \lra
Rx\otimes_RN$ for which $\psi_N(g) = x\otimes g$ for all $g \in N$
is injective.

For discussion of tight closure we use $R^{\circ}$ to denote the
complement in $R$ of the union of the minimal prime ideals of $R$.
Let $L$ and $M$ be $R$-modules and let $K$ be a submodule of $L$.
Observe that there is a natural structure as $\nn$-graded left
$R[x,f]$-module on $R[x,f]\otimes_RM =
\bigoplus_{n\in\nn}(Rx^n\otimes_RM)$.  An element $m \in M$ belongs
to $0^*_M$, the {\em tight closure of the zero submodule in $M$\/},
if and only if there exists $c \in R^{\circ}$ such that the element
$1 \otimes m \in (R[x,f]\otimes_RM)_0$ is annihilated by $cx^j$ for
all $j \gg 0$: see M. Hochster and C. Huneke \cite[\S 8]{HocHun90}.
Furthermore, the tight closure $K^*_L$ of $K$ in $L$ is the inverse
image, under the natural epimorphism $L \lra L/K$, of $0^*_{L/K}$,
the tight closure of $0$ in $L/K$.

A test element (for modules) for $R$ is an element $c \in R^{\circ}$
such that, for every finitely generated $R$-module $M$ and every $j
\in \nn$, the element $cx^j$ annihilates $1 \otimes m \in
(R[x,f]\otimes_RM)_0$ for every $m \in 0^*_M$. If $R$ has a test
element, then it must be reduced. It is a result of Hochster and
Huneke \cite[Theorem (6.1)(b)]{HocHun94} that a reduced algebra of
finite type over an excellent local ring of characteristic $p$ has a
test element.

In attempts to gain a greater understanding of tight closure, the
author has investigated in the recent paper \cite{ga} properties of
certain left modules over the Frobenius skew polynomial ring
$R[x,f]$. Let $H$ be a left $R[x,f]$-module. The {\em graded
annihilator $\grann_{R[x,f]}H$\/} (or grann$_{R[x,f]}H$) {\em of
$H$\/} is defined in \cite[1.5]{ga} and is the largest graded
two-sided ideal of $R[x,f]$ that annihilates $H$. We shall use
$\mathcal{G}(H)$ (or $\mathcal{G}_{R[x,f]}(H)$ when it is desirable
to emphasize the ring $R$) to denote the set of all graded
annihilators of $R[x,f]$-submodules of $H$.

When $H$ is $x$-torsion-free, $\grann_{R[x,f]}H$ has the form $\fb
R[x,f] = \bigoplus _{n\in\nn} \fb x^n$ for some radical ideal $\fb$
of $R$, and, in that case, we write $\mathcal{I}(H)$ for the set of
(necessarily radical) ideals $\fc$ of $R$ for which there is an
$R[x,f]$-submodule $N$ of $H$ such that $\grann_{R[x,f]}N = \fc
R[x,f]$; in these circumstances, the members of $\mathcal{I}(H)$ are
referred to as the {\em $H$-special $R$-ideals}, and we have
$\mathcal{G}_{R[x,f]}(H) = \left\{\fb R[x,f] : \fb \in
\mathcal{I}(H)\right\}$. One of the main results of \cite{ga} is
Corollary 3.11, which states that $\mathcal{I}(H)$ is finite if $H$
is ($x$-torsion-free and) Artinian (or Noetherian) as an $R$-module.

In \cite[\S 4]{ga}, the author applied these ideas in the case where
$R$ is an $F$-injective Gorenstein local ring, with maximal ideal
$\fm$ and dimension $d$, to the `top' local cohomology module $H :=
H^d_{\fm}(R)$ of $R$. (Recall that every local cohomology module of
$R$ has a natural structure as a left $R[x,f]$-module.) The
statement that $R$ is $F$-injective implies that $H$, with its
natural structure as a left $R[x,f]$-module, is $x$-torsion-free,
and this implies that $R$ must be reduced. Note also that, in this
case, $H^d_{\fm}(R) \cong E_R(R/\fm)$. The ideas of \cite[\S 4]{ga}
yield the finite set $\mathcal{I}(H)$ of radical ideals of $R$. Let
$\fb$ denote the smallest ideal of positive height in
$\mathcal{I}(H)$ (interpret $\height R$ as $\infty$). In
\cite[Corollary 4.7]{ga} it was shown that, if $c$ is any element of
$\fb\cap R^{\circ}$, then $c$ is a test element for $R$, and that
$\fb$ is the test ideal $\tau(R)$ of $R$ (that is (in this case),
the ideal of $R$ generated by all test elements of $R$). It should
be noted that these results were obtained without the assumption
that $R$ is excellent.

This paper partially generalizes the above-described results of
\cite[Corollary 4.7]{ga} to the case where $R$ is local, with
maximal ideal $\fm$, and $E := E_R(R/\fm)$ carries a structure of
$x$-torsion-free left $R[x,f]$-module. In this situation, which is
more general than that of \cite[Corollary 4.7]{ga}, we again use
results from \cite[\S 4]{ga} that yield the finite set
$\mathcal{I}(E)$ of radical ideals of $R$. One of the main results
of this paper is that, if $\fb$ is the smallest ideal of positive
height in $\mathcal{I}(E)$, then each element of $\fb \cap
R^{\circ}$ is a test element (for modules) for $R$.

A byproduct is an analogue of a result of Janet Cowden Vassilev in
\cite[\S 3]{Cowde98}: she showed, in the case where $R$ is an
$F$-pure homomorphic image of an $F$-finite regular local ring, that
there exists a strictly ascending chain $0 = \tau_0 \subset \tau_1
\subset \cdots \subset \tau_t = R$ of radical ideals of $R$ such
that, for each $i = 0, \ldots, t-1$, the reduced local ring
$R/\tau_i$ is $F$-pure and its test ideal (has positive height and)
is exactly $\tau_{i+1}/\tau_i$. This paper presents an analogous
result in the case where $R$ is complete (but not necessarily
$F$-finite) and $E$ has a structure as an $x$-torsion-free left
$R[x,f]$-module.

\section{Graded left modules over the Frobenius skew polynomial ring}

This paper builds on the results of \cite{ga}, and we shall use the
notation and terminology of \S 1 of that paper.

\begin{ntn}
\label{nt.1} The notation introduced in the Introduction will be
maintained.

The symbols $\fa$ and $\fb$ will always denote ideals of $R$. We
shall only assume that $R$ is local when this is explicitly stated;
then, the notation `$(R,\fm)$' will denote that $\fm$ is the maximal
ideal of $R$.

Let $H$ be a left $R[x,f]$-module. Recall from \cite[1.5]{ga} that
an $R[x,f]$-submodule of $H$ is said to be a {\em special
annihilator submodule of $H$\/} if it has the form $\ann_H(\fB)$ for
some {\em graded\/} two-sided ideal $\fB$ of $R[x,f]$. As in
\cite{ga}, we shall use $\mathcal{A}(H)$ to denote the set of
special annihilator submodules of $H$.

The {\em $x$-torsion submodule $\Gamma_x(H)$ of $H$,\/} and the
concept of {\em $x$-torsion-free left $R[x,f]$-module,\/} were
defined in \cite[1.2]{ga}. The left $R[x,f]$-module $H/\Gamma_x(H)$
is automatically $x$-torsion-free.

For $n \in \Z$, we shall denote the $n$th component of a $\Z$-graded
left $R[x,f]$-module $G$ by $G_n$. If $\phi : L \lra M$ is a
homogeneous homomorphism of $\Z$-graded left $R[x,f]$-modules (of
degree $0$), then the notation $\phi = \bigoplus_{n\in\Z}\phi_n:
\bigoplus_{n\in\Z}L_n \lra \bigoplus_{n\in\Z}M_n$ will indicate that
$\phi_n : L_n \lra M_n$ is the restriction of $\phi$ to $L_n$ (for
all $n \in\Z$). For $t \in \Z$, we shall denote the {\em $t$th shift
functor\/} on the category of ($\Z$-)graded left $R[x,f]$-modules
and homogeneous homomorphisms (of degree $0$) by $ (\:
{\scriptscriptstyle \bullet} \:)(t)$: thus, for a graded left
$R[x,f]$-module $M = \bigoplus _{n \in \Z} M_{n}$, we have
$(M(t))_{n} = M_{n+t}$ for all $n \in \Z$; also, $f(t)\lceil\:
_{(M(t))_{n}} = f\lceil\: _{M_{n+t}}$ for each morphism $f$ in the
above-mentioned category and all $n \in \Z$\@.

\end{ntn}

In this paper, much use will be made of results in \S 1 and \S 3 of
\cite{ga}. There now follow brief reminders of some of those results
that are particularly important for this paper.

\begin{lem}[{\cite[Lemma 1.7(v)]{ga}}]\label{ga1.7} There
is an order-reversing bijection, $\Gamma\/,$ from the set
$\mathcal{A}(H)$ of special annihilator submodules of $H$ to the set
$\mathcal{G}(H)$ of graded annihilators of submodules of $H$ given
by
$$
\Gamma : N \longmapsto \grann_{R[x,f]}N.
$$
The inverse bijection, $\Gamma^{-1},$ also order-reversing, is given
by
 $$
\Gamma^{-1} : \fB \longmapsto \ann_H(\fB).
$$
\end{lem}

\begin{lem}[{\cite[Lemma 1.9]{ga}}]\label{ga1.9} Let $H$ be an $x$-torsion-free
left $R[x,f]$-module. Then there is a radical ideal $\fb$ of $R$
such that $\grann_{R[x,f]}H = \fb R[x,f] = \bigoplus _{n\in\nn} \fb
x^n.$
\end{lem}

\begin{prop}[{\cite[Proposition 1.11]{ga}}]
\label{ga1.11} Let $H$ be an $x$-torsion-free left $R[x,f]$-module.

There is an order-reversing bijection, $\Delta : \mathcal{A}(H) \lra
\mathcal{I}(H),$ from the set $\mathcal{A}(H)$ of special
annihilator submodules of $H$ to the set $\mathcal{I}(H)$ of
$H$-special $R$-ideals given by \[ \Delta : N \longmapsto
\left(\grann_{R[x,f]}N\right)\cap R = (0:_RN). \]

The inverse bijection, $\Delta^{-1} : \mathcal{I}(H) \lra
\mathcal{A}(H),$ also order-reversing, is given by \[ \Delta^{-1} :
\fb \longmapsto \ann_H\left(\fb R[x,f])\right). \] When $N \in
\mathcal{A}(H)$ and $\fb \in \mathcal{I}(H)$ are such that
$\Delta(N) = \fb$, we shall say simply that `{\em $N$ and $\fb$
correspond}'.
\end{prop}

\begin{cor}[{\cite[Corollary 3.7]{ga}}]\label{ga3.7}
Let $H$ be an $x$-torsion-free left $R[x,f]$-module. Then the set of
$H$-special $R$-ideals is precisely the set of all finite
intersections of prime $H$-special $R$-ideals (provided one includes
the empty intersection, $R$, which corresponds to the zero special
annihilator submodule of $H$). In symbols, \[ \mathcal{I}(H) =
\left\{ \fp_1 \cap \ldots \cap \fp_t : t \in \nn \mbox{~and~} \fp_1,
\ldots, \fp_t \in \mathcal{I}(H)\cap\Spec(R)\right\}. \]
\end{cor}

\begin{thm}[{\cite[Corollary 3.11]{ga}}]
\label{ga3.11} Suppose that $H$ is an $x$-torsion-free left
$R[x,f]$-module that is either Artinian or Noetherian as an
$R$-module. Then the set $\mathcal{I}(H)$ of $H$-special $R$-ideals
is finite.
\end{thm}

By Proposition \ref{ga1.11}, for an $x$-torsion-free left
$R[x,f]$-module $H$, the set $\mathcal{I}(H)$ is finite if and only
if the set $\mathcal{A}(H)$ of special annihilator submodules of $H$
is finite.

The terminology in the following definition is inspired by the
Hartshorne--Speiser--Lyubeznik Theorem: see Lyubeznik
\cite[Proposition 4.4]{Lyube97} and compare R. Hartshorne and R.
Speiser \cite[Proposition 1.11]{HarSpe77}.

\begin{defi}[{\cite[Definition 3.15]{ga}}]
\label{ga3.15} Let $H$ be a left $R[x,f]$-module. We say that $H$
{\em admits an HSL-number\/} if there exists $e \in \nn$ such that
$x^e\Gamma_x(H) = 0$; then we call the smallest such $e$ the {\em
HSL-number\/} of $H$, and denote this by $\HSL(H)$.

If there is no such non-negative integer, that is, if $H$ does not
admit an $\HSL$-number, then we shall write $\HSL(H) = \infty$.
\end{defi}

In this terminology, the conclusion of the
Hartshorne--Speiser--Lyubeznik Theorem is that, when $(R,\fm)$ is
local, a left $R[x,f]$-module that is Artinian as an $R$-module
admits an $\HSL$-number.

We end this section by showing that the results from \cite{ga}
quoted above lead quickly to a generalization of a result of F.
Enescu and M. Hochster \cite[Theorem 3.7]{EneHocarX}.

\begin{prop}\label{gor.3} Suppose that the local ring $(R,\fm)$ is
complete, and suppose that $E := E_R(R/\fm)$ has a structure as left
$R[x,f]$-module. Then every $R[x,f]$-submodule of $E$ is a special
annihilator submodule.

Consequently, if $E$ is $x$-torsion-free, then there are only
finitely many $R[x,f]$-submodules of $E$, and, for each $E$-special
$R$-ideal $\fb$, we have $\ann_E\left(\bigoplus_{n\in\nn}\fb
x^n\right) = \ann_E(\fb)$.
\end{prop}

\begin{proof} Let $L$ be an $R[x,f]$-submodule of $E$, and let $\fB$
be its graded annihilator. Thus there is an increasing sequence
$(\fb_n)_{n \in \nn}$ of ideals of $R$ such that $\fB =
\bigoplus_{n\in\nn}\fb_nx^n$, and $\fb_0 = (0:_RL)$. Now, by Matlis
duality (see, for example, \cite[p.\ 154]{SV}), every $R$-submodule
$M$ of $E$ satisfies $M = \ann_E((0:_RM))$. Therefore
\begin{align*}
L & \subseteq \ann_E(\grann_{R[x,f]}L)\\ & = \ann_E\left({\textstyle
\bigoplus_{n\in\nn}\fb_nx^n}\right) \subseteq \ann_E(\fb_0) =
\ann_E((0:_RL)) = L.
\end{align*}
Therefore $L = \ann_E(\grann_{R[x,f]}L) =
\ann_E\left(\bigoplus_{n\in\nn}\fb_nx^n\right) = \ann_E(\fb_0)$ and
$L$ is a special annihilator submodule of $E$.

For the claims in the final paragraph, note that it follows from the
above that \[\ann_E\left({\textstyle \bigoplus_{n\in\nn}\fb
x^n}\right) = \ann_E(\fb),\] and, by \cite[Corollary 3.11]{ga} and
the fact that $E$ is Artinian as an $R$-module, there are only
finitely many special annihilator submodules of $E$.
\end{proof}

\begin{cor}\label{gor.4} Suppose that $(R,\fm)$ is local, and that $E := E_R(R/\fm)$ has a structure as
an $x$-torsion-free left $R[x,f]$-module.  Then there are only
finitely many $R[x,f]$-submodules of $E$.
\end{cor}

\begin{proof} Denote the completion of $R$ by $(\widehat{R},\widehat{\fm})$.
Recall the natural $\widehat{R}$-module structure on the Artinian
$R$-module $E$: given $h \in E$, there exists $t \in \N$ such that
$\fm^th = 0$; for an $\widehat{r} \in \widehat{R}$, choose any $r
\in R$ such that $\widehat{r} - r \in \fm^t\widehat{R}$; then
$\widehat{r}h = rh$. It is easy to see from this that $x\widehat{r}h
= \widehat{r}^pxh$ for all $h \in E$ and $\widehat{r} \in
\widehat{R}$; thus $E$ inherits a structure as left
$\widehat{R}[x,f]$-module that extends both its $R[x,f]$-module and
$\widehat{R}$-module structures. In particular, $E$ is
$x$-torsion-free as left $\widehat{R}[x,f]$-module.

It is easy to use \cite[10.2.10]{LC} (for example) to see that, when
$E$ is regarded as an $\widehat{R}$-module in this way, then it is
isomorphic to $E_{\widehat{R}}(\widehat{R}/\widehat{\fm})$. Since a
subset of $E$ is an $R[x,f]$-submodule if and only if it is an
$\widehat{R}[x,f]$-submodule, it is enough for us to prove the claim
under the additional assumption that $R$ is complete; in that case,
the desired conclusion follows from Proposition \ref{gor.3}.
\end{proof}

As a corollary, we obtain a result that has already been established
by F. Enescu and M. Hochster. Recall that the $d$-dimensional local
ring $(R,\fm)$ is said to be {\em quasi-Gorenstein\/} precisely when
the top local cohomology module $H^d_{\fm}(R)$ is isomorphic to
$E_R(R/\fm)$.

\begin{cor}[{Enescu--Hochster \cite[Theorem 3.7]{EneHocarX}}]\label{gor.4a}
Suppose that the $d$-dimensional local ring
$(R,\fm)$ is quasi-Gorenstein and $F$-pure. Then $H^d_{\fm}(R)$,
regarded as an $R[x,f]$-module in the natural way, has only finitely
many $R[x,f]$-submodules.

In particular, if $(R,\fm)$ is Gorenstein and $F$-pure, then
$H^d_{\fm}(R)$ has only finitely many $R[x,f]$-submodules.
\end{cor}

\begin{proof} This follows from Corollary \ref{gor.4} because $H^d_{\fm}(R) \cong
E_R(R/\fm)$ and the hypothesis that $R$ is $F$-pure ensures that the
natural $R[x,f]$-module structure on $H^d_{\fm}(R)$ is
$x$-torsion-free.
\end{proof}

\section{New left $R[x,f]$-modules from old}\label{new}

The purpose of this section is to introduce some methods for the
construction of new left $R[x,f]$-modules from old, and to describe
the sets of graded annihilators, and the $\HSL$-numbers, of (some
of) these new modules in terms of the corresponding invariants of
the original modules. Although the main applications of these ideas
in this paper will be to left $R[x,f]$-modules which are
$x$-torsion-free, it is convenient to note at the same time
conclusions that apply in the more general case.

\begin{lem}\label{nt.3} Let
$\left(H^{(\lambda)}\right)_{\lambda\in\Lambda}$ be a non-empty
family of $\Z$-graded left $R[x,f]$-modules, with gradings given by
$H^{(\lambda)}= \bigoplus_{n\in\Z}H^{(\lambda)}_n$ for each
$\lambda\in\Lambda$. For each $n \in \Z$, set $H_n :=
\prod_{\lambda\in\Lambda}H^{(\lambda)}_n$. Then the $R$-module
\[
H := \bigoplus_{n\in\Z}H_n =
\bigoplus_{n\in\Z}\left(\prod_{\lambda\in\Lambda}H^{(\lambda)}_n\right)
\]
has a natural structure as a ($\Z$-graded) left $R[x,f]$-module in
which \[x\big(h^{(\lambda)}_n\big)_{\lambda\in\Lambda} =
\big(xh^{(\lambda)}_n\big)_{\lambda\in\Lambda}\in
\prod_{\lambda\in\Lambda}H^{(\lambda)}_{n+1} \quad \mbox{for all~}
\big(h^{(\lambda)}_n\big)_{\lambda\in\Lambda} \in
\prod_{\lambda\in\Lambda}H^{(\lambda)}_n.\] This graded left
$R[x,f]$-module $H$ is, in fact, the product of
$\left(H^{(\lambda)}\right)_{\lambda\in\Lambda}$ in the category of
$\Z$-graded left $R[x,f]$-modules and homogeneous
$R[x,f]$-homomorphisms\/ {\rm (}of degree $0${\rm )}; however, to
avoid possible confusion, we shall denote the module $H$ by
$\prod^{\prime}_{\lambda\in\Lambda}H^{(\lambda)}$.
\end{lem}

\begin{proof} This is straightforward and will be left to the
reader; one can use \cite[Lemma 1.3]{KS} to facilitate the
verification that $H$ has a structure as a left $R[x,f]$-module as
claimed.
\end{proof}

\begin{rmk}\label{nt.4} Let
$\left(H^{(\lambda)}\right)_{\lambda\in\Lambda}$ be a non-empty
family of $\Z$-graded left $R[x,f]$-modules, and, as in Lemma
\ref{nt.3}, set $H :=
\prod^{\prime}_{\lambda\in\Lambda}H^{(\lambda)}$. Let $\fB$ be a
graded two-sided ideal of $R[x,f]$. Then it is straightforward to
check that
\begin{enumerate}
\item $\ann_H\fB = \prod^{\prime}_{\lambda\in\Lambda}\ann_{H^{(\lambda)}}\fB$, and
\item $\grann_{R[x,f]}H = \bigcap_{\lambda\in\Lambda}
\grann_{R[x,f]}H^{(\lambda)}$.
\end{enumerate}
\end{rmk}

\begin{lem}\label{nt.5} Let
$\left(H^{(\lambda)}\right)_{\lambda\in\Lambda}$ be a non-empty
family of $\Z$-graded left $R[x,f]$-modules, and, as in Lemma\/ {\rm
\ref{nt.3}}, set $H :=
\prod^{\prime}_{\lambda\in\Lambda}H^{(\lambda)}$.

\begin{enumerate}
\item We have $\mathcal{G}(H) = \left\{ \bigcap_{\lambda \in \Lambda}\fB_{\lambda} : \fB_{\lambda} \in
\mathcal{G}(H^{(\lambda)}) \mbox{~for all~} \lambda \in
\Lambda\right\}$.
\item Consequently, if there exists a set $\mathcal{G}'$ of graded two-sided
ideals of $R[x,f]$ such that $\mathcal{G}(H^{(\lambda)}) =
\mathcal{G}'$ for all $\lambda \in \Lambda$, then $\mathcal{G}(H) =
\mathcal{G}'$.
\item We have $\HSL(H) = \sup \left\{ \HSL(H^{(\lambda)}) : \lambda \in
\Lambda\right\}$ (even if $\HSL(H^{(\mu)}) = \infty$ for some $\mu
\in \Lambda$). Thus $H$ is $x$-torsion-free if and only if
$H^{(\lambda)}$ is $x$-torsion-free for all $\lambda\in\Lambda$.
\item If $\HSL(H)$ is finite, then $\Gamma_x(H) =
\prod^{\prime}_{\lambda\in\Lambda}\Gamma_x(H^{(\lambda)})$ and there
is a homogeneous isomorphism of graded left $R[x,f]$-modules
\[
H/\Gamma_x(H) \stackrel{\cong}{\lra}
\prod_{\lambda\in\Lambda}{\textstyle ^{^{^{\Large
\prime}}}}H^{(\lambda)}/\Gamma_x(H^{(\lambda)}).
\]
\end{enumerate}
\end{lem}

\begin{proof} (i) Let $\fB \in \mathcal{G}(H)$. Then, by
\ref{ga1.7} and \ref{nt.4}, we have
\begin{align*} \fB & = \grann_{R[x,f]}(\ann_H(\fB)) =
\grann_{R[x,f]}\left(\prod_{\lambda\in\Lambda}{\textstyle
^{^{^{\Large
\prime}}}}\ann_{H^{(\lambda)}}\fB\right)\\
& = \bigcap_{\lambda\in\Lambda}
\grann_{R[x,f]}(\ann_{H^{(\lambda)}}\fB)\\
& \in \left\{ \bigcap_{\lambda \in \Lambda}\fB_{\lambda} :
\fB_{\lambda} \in \mathcal{G}(H^{(\lambda)}) \mbox{~for all~}
\lambda \in \Lambda\right\}.
\end{align*}

Next, for each $\lambda \in \Lambda$, let $\fB_{\lambda} \in
\mathcal{G}(H^{(\lambda)})$. Then $\fB_{\lambda} =
\grann_{R[x,f]}(\ann_{H^{(\lambda)}}(\fB_{\lambda}))$, by
\ref{ga1.7}. Set $L :=
\prod^{\prime}_{\lambda\in\Lambda}\ann_{H^{(\lambda)}}(\fB_{\lambda})$,
a graded $R[x,f]$-submodule of $H$. Then, by \ref{nt.4},
\[
\grann_{R[x,f]}L = \bigcap_{\lambda\in\Lambda}
\grann_{R[x,f]}(\ann_{H^{(\lambda)}}(\fB_{\lambda})) =
\bigcap_{\lambda\in\Lambda} \fB_{\lambda},
\]
and so $\bigcap_{\lambda\in\Lambda} \fB_{\lambda} \in
\mathcal{G}(H)$.

(ii) This is now immediate from part (i), because each
$\mathcal{G}(H^{(\lambda)})$ (for a $\lambda \in \Lambda$) is closed
under taking arbitrary intersections.

(iii),(iv) These are straightforward, and left to the reader.
\end{proof}

\begin{lem}\label{nt.6} Let $H$ be a left $R[x,f]$-module.
Let $\fB$ be a two-sided ideal of $R[x,f]$. For all $n \in \nn$, set
$H_n := H$. Then the $R$-module $\widetilde{H} :=
\bigoplus_{n\in\nn} H_n$ has a natural structure as a graded left
$R[x,f]$-module under which the result of multiplying $h_n \in H_n =
H$ on the left by $x$ is the element $xh_n \in H_{n+1} = H$.

Furthermore,
\begin{enumerate}
\item $\grann_{R[x,f]}(\widetilde{H}) = \grann_{R[x,f]}(H)$, and
\item $
\ann_{\widetilde{H}}\fB = \bigoplus_{n\in\nn} \ann_{H_n}\fB =
\ann_{H}\fB \oplus \ann_{H}\fB \oplus \cdots \oplus \ann_{H}\fB
\oplus \cdots = \widetilde{\ann_H\fB}$.

\end{enumerate}
\end{lem}

\begin{proof} One can use \cite[Lemma 1.3]{KS} to facilitate the
verification that $\widetilde{H}$ has a structure as a left
$R[x,f]$-module as claimed. The claims in (i) and (ii) are clear.
\end{proof}

\begin{lem}\label{nt.6a} Let $H$ be a left $R[x,f]$-module; set $G := H/\Gamma_x(H)$.
Let $\widetilde{H}$ be the graded left $R[x,f]$-module constructed
from $H$ as in Lemma\/ {\rm \ref{nt.6}}. Then
\begin{enumerate}
\item $\mathcal{G}(\widetilde{H}) = \mathcal{G}(H)$;
\item $\HSL(\widetilde{H}) = \HSL(H)$ and $\Gamma_x(\widetilde{H}) =
\bigoplus_{n\in\nn}\Gamma_x(H_n) = \widetilde{\Gamma_x(H)}$; and
\item there is an $R[x,f]$-isomorphism
\[
\widetilde{H}/\Gamma_x(\widetilde{H}) \cong \bigoplus_{n\in\nn}
H_n/\Gamma_x(H_n) = \widetilde{G}.
\]
\end{enumerate}
\end{lem}

\begin{proof} (i) Let $\fB \in \mathcal{G}(\widetilde{H})$. Then, by Lemmas \ref{ga1.7}
and \ref{nt.6}, we have \[ \fB =
\grann_{R[x,f]}(\ann_{\widetilde{H}}\fB) =
\grann_{R[x,f]}(\widetilde{\ann_H\fB}) = \grann_{R[x,f]}(\ann_H\fB)
\in \mathcal{G}(H).\] On the other hand, if $\fB \in
\mathcal{G}(H)$, then, again by Lemmas \ref{ga1.7} and \ref{nt.6},
\[
\fB = \grann_{R[x,f]}(\ann_H\fB) =
\grann_{R[x,f]}(\widetilde{\ann_H\fB}) =
\grann_{R[x,f]}(\ann_{\widetilde{H}}\fB) \in
\mathcal{G}(\widetilde{H}).
\]

(ii),(iii) These are straightforward and left to the reader.
\end{proof}

\begin{rmk}\label{nt.6b} Suppose that $(R,\fm)$ is local. Let $H := H^d_{\fm}(R)$, considered as a left
$R[x,f]$-module in the natural way (recalled in \cite[Reminder
4.1(ii)]{ga}). The isomorphisms described in \cite[Remark
4.2(iii)]{ga} show that, if we apply the construction of Lemma
\ref{nt.6} to this $H$, then the resulting graded left
$R[x,f]$-module $\widetilde{H}$ is isomorphic to
$R[x,f]\otimes_RH^d_{\fm}(R)$.
\end{rmk}

\begin{ntn}\label{nt.6n} For $t \in \N$, we refer to the
mapping $f : R^t \lra R^t$ for which $f((r_1, \ldots, r_t)) =
(r_1^p, \ldots, r_t^p)$ for all $(r_1, \ldots, r_t) \in R^t$ as the
{\em Frobenius map\/}.
\end{ntn}

\begin{lem}\label{nt.8} Let $b \in \N$ and $W = \bigoplus_{n\geq b}W_n$ be a $\Z$-graded
left $R[x,f]$-module; let $g_1, \ldots, g_t \in W_b$, and let
$$
K := \left\{ (r_1, \ldots, r_t) \in R^t : {\textstyle
\sum_{i=1}^tr_ig_i = 0}\right\}.
$$

Then there is a graded left $R[x,f]$-module
\[
W' = \bigoplus_{n\geq b-1}W'_n = \left(R^t/f^{-1}(K)\right)\oplus
W_b \oplus W_{b+1} \oplus \cdots \oplus W_i \oplus \cdots
\]
(so that $W'_{b-1} = R^t/f^{-1}(K)$ and $W'_n = W_n$ for all $n \geq
b$) which has $W$ as an $R[x,f]$-submodule and for which $x((r_1,
\ldots, r_t) + f^{-1}(K)) = \sum_{i=1}^tr_i^pg_i$ for all $(r_1,
\ldots, r_t) \in R^t$. We call $W'$ the\/ {\em $1$-place extension
of $W$ by $g_1, \ldots, g_t$}, and denote it by $$\ext(W;g_1,
\ldots, g_t;1).$$

If $W$ is $x$-torsion-free, then so too is $\ext(W;g_1, \ldots,
g_t;1)$, and then $$\mathcal{G}(\ext(W;g_1, \ldots, g_t;1)) =
\mathcal{G}(W)$$ and $\mathcal{I}(\ext(W;g_1, \ldots, g_t;1)) =
\mathcal{I}(W).$
\end{lem}

\begin{proof} One can use \cite[Lemma 1.3]{KS} to facilitate the
verification that $W'$ is a graded left $R[x,f]$-module.

Suppose that $W$ is $x$-torsion-free, and let $(r_1, \ldots, r_t)
\in R^t$ be such that the element $\eta := (r_1, \ldots, r_t) +
f^{-1}(K)$ of $W'_{b-1}$ belongs to $\Gamma_x(W')$. Then there
exists $h \in \N$ such that $x^h\eta = x^{h-1}(\sum_{i=1}^tr_i^pg_i)
= 0$. Since $W$ is $x$-torsion-free, we have $\sum_{i=1}^tr_i^pg_i =
x\eta = 0$. Therefore $f((r_1, \ldots, r_t)) \in K$, and so $\eta :=
(r_1, \ldots, r_t) + f^{-1}(K) = 0$. It follows that, if $W$ is
$x$-torsion-free, then $W'$ is $x$-torsion-free. The converse is
clear.

In order to prove the final two claims, it is sufficient to prove
that $\mathcal{I}(W') = \mathcal{I}(W)$. Since $W$ is an
$R[x,f]$-submodule of $W'$, it is clear that $\mathcal{I}(W)
\subseteq \mathcal{I}(W')$. Let $\fb \in \mathcal{I}(W')$, so that
$\fb R[x,f]$ is the graded annihilator of the graded $R[x,f]$
submodule $L' := \bigoplus_{n \geq b-1}L'_n = \ann_{W'}(\fb R[x,f])$
of $W'$. Set $L := \bigoplus_{n \geq b}L'_n$, and let
$\grann_{R[x,f]}L = \fc R[x,f]$, where $\fc$ is a radical ideal of
$R$. Note that $\fc \in \mathcal{I}(W)$ and $\fb \subseteq \fc$. Now
the two-sided ideal $\bigoplus_{n\geq 1} \fc x^n$ of $R[x,f]$
annihilates $L'$, so that $xc m = c^p xm=0$ for all $c \in \fc$ and
$m \in L'$. Since $L'$ is $x$-torsion-free by the above, it follows
that $\fc L' = 0$ and $\bigoplus_{n\geq 0} \fc x^n = \fc R[x,f]$
annihilates $L'$. Therefore $\fc \subseteq \fb$, and $\fb =
\fc\in\mathcal{I}(W)$.
\end{proof}

\begin{rmk}\label{nt.8a} Here we use
the notation and terminology of Lemma \ref{nt.8}, and write
$\overline{W}$ for $W/\Gamma_x(W)$, and also use `overlines' to
denote natural images in $\overline{W}$ of elements of $W$.

It is straightforward to check that there is a homogeneous
isomorphism of graded left $R[x,f]$-modules
\[
\ext(W;g_1, \ldots, g_t;1)/\Gamma_x(\ext(W;g_1, \ldots, g_t;1))
\cong \ext(\overline{W};\overline{g_1}, \ldots, \overline{g_t};1),
\] so that, by Lemma \ref{nt.8}, $$
\mathcal{G}\big(\ext(W;g_1, \ldots, g_t;1)/\Gamma_x(\ext(W;g_1,
\ldots, g_t;1))\big) = \mathcal{G}(\overline{W})$$ and $$
\mathcal{I}\big(\ext(W;g_1, \ldots, g_t;1)/\Gamma_x(\ext(W;g_1,
\ldots, g_t;1))\big) = \mathcal{I}(\overline{W}).$$
\end{rmk}

\begin{defi}\label{nt.9} Let $b \in \N$ and $W = \bigoplus_{n\geq b}W_n$ be a $\Z$-graded
left $R[x,f]$-module; let $g_1, \ldots, g_t \in W_b$. The $1$-place
extension $\ext(W;g_1, \ldots, g_t;1)$ of $W$ by $g_1, \ldots, g_t$
was defined in Lemma \ref{nt.8}. Recall that we defined $$K :=
\left\{ (r_1, \ldots, r_t) \in R^t : {\textstyle \sum_{i=1}^tr_ig_i
= 0}\right\}.
$$

Now let $h \in \N$ with $h \geq 2$. The {\em $h$-place extension
$\ext(W;g_1, \ldots, g_t;h)$ of $W$ by $g_1, \ldots, g_t$\/} is the
graded left $R[x,f]$-module
\[ \left(R^t/f^{-h}(K)\right)\oplus \cdots \oplus
\left(R^t/f^{-1}(K)\right)\oplus W_b \oplus  \cdots \oplus W_i
\oplus \cdots
\]
which has $\ext(W;g_1, \ldots, g_t;1)$ as a graded
$R[x,f]$-submodule and is such that $$x(v+ f^{-j}(K)) = f(v)+
f^{-(j-1)}(K) \quad \text{for all~} v \in R^t \text{~and~} j =
h,h-1, \ldots, 2.
$$

For each $i \in \{1, \ldots, t\}$, let $e_i$ denote the element
$(0,\ldots,0,1,0,\ldots,0)$ of $R^t$ which has a $1$ in the $i$th
spot and all other components $0$. It is straightforward to check
that
$$
\ext(W;g_1, \ldots, g_t;h) = \ext(\ext(W;g_1, \ldots,
g_t;1);\overline{e_1},\ldots,\overline{e_t};h-1),$$ where, for $v
\in R^t$, we use $\overline{v}$ to denote $v + f^{-1}(K)$, and
$$
\ext(W;g_1, \ldots, g_t;h) = \ext(\ext(W;g_1, \ldots,
g_t;h-1);\widetilde{e_1},\ldots,\widetilde{e_t};1),$$ where, for $v
\in R^t$, we use $\widetilde{v}$ to denote $v + f^{-(h-1)}(K)$.

It is a consequence of Lemma \ref{nt.8} that, if $W$ is
$x$-torsion-free, then so too is $\ext(W;g_1, \ldots, g_t;h)$, and
then $\mathcal{G}(\ext(W;g_1, \ldots, g_t;h)) = \mathcal{G}(W)$ and
$$\mathcal{I}(\ext(W;g_1, \ldots, g_t;h)) = \mathcal{I}(W).$$
\end{defi}

\begin{prop}\label{ext.23} Let $b \in \N$ and $W = \bigoplus_{n\geq b}W_n$ be a $\Z$-graded
left $R[x,f]$-module, and let $M$ be an $R$-module generated by the
finite set $\{m_1, \ldots, m_t\}$; then we can form the graded
$R[x,f]$-submodule $\bigoplus_{i \geq b}(Rx^i\otimes_RM)$ of
$R[x,f]\otimes_RM$. Suppose that there is given a homogeneous
$R[x,f]$-homomorphism $\lambda' = \bigoplus_{i\geq b}\lambda_i :
\bigoplus_{i \geq b}(Rx^i\otimes_RM) \lra W$.

For each $j = 1, \ldots, t$, let $g_j := \lambda_b(x^b\otimes m_j)
\in W_b$. Set $$ K := \left\{ (r_1, \ldots, r_t) \in R^t :
{\textstyle \sum_{j=1}^tr_jg_j = 0}\right\},
$$
as in Lemma\/ {\rm \ref{nt.8}}. For each $i = 0, 1, \ldots, b-1$,
there exists an $R$-homomorphism $\lambda_i : Rx^i \otimes_R M \lra
R^t/f^{-(b-i)}(K)$ such that
$$
\lambda_i\left({\textstyle \sum_{j=1}^t r_jx^i\otimes m_j}\right) =
(r_1,\ldots,r_t) + f^{-(b-i)}(K) \quad \text{for all~} r_1,
\ldots,r_t \in R.
$$
Furthermore, $$ \lambda := \bigoplus_{i\in \nn}\lambda_i :
R[x,f]\otimes_RM = \bigoplus_{i\in \nn}(Rx^i\otimes_RM) \lra
\ext(W;g_1, \ldots, g_t;b)
$$
is a homogeneous $R[x,f]$-homomorphism that extends $\lambda'$.
\end{prop}

\begin{proof} This is straightforward once it is been noted that, if
$i \in \{0,\ldots,b-1\}$ and $v = (r_1, \ldots,r_t),w =(s_1,
\ldots,s_t) \in R^t$ are such that $\sum_{j=1}^tr_jx^i \otimes m_j =
\sum_{j=1}^ts_jx^i \otimes m_j$, then $\sum_{j=1}^tr_j^{p^{b-i}}x^b
\otimes m_j = \sum_{j=1}^ts_j^{p^{b-i}}x^b \otimes m_j$, so that $v
- w \in f^{-(b-i)}(K)$ because
\begin{align*} {\textstyle \sum_{j=1}^t
(r_j-s_j)^{p^{b-i}}g_j} &= {\textstyle \sum_{j=1}^t
(r_j-s_j)^{p^{b-i}}\lambda_b(x^b\otimes m_j)}\\ &= {\textstyle
\sum_{j=1}^t \lambda_b\left((r_j-s_j)^{p^{b-i}}x^b\otimes
m_j\right)}= 0.
\end{align*}
\end{proof}

\section{\sc Use of an $R[x,f]$-module structure on the injective
envelope of the simple module over a local ring}

\begin{lem}\label{gor.9} Suppose that $(R,\fm)$ is local and that there exists a left
$R[x,f]$-module $E$ which, as $R$-module, is isomorphic to
$E_R(R/\fm)$, the injective envelope of the simple $R$-module
$R/\fm$. Construct the graded left $R[x,f]$-module $\widetilde{E}$
from $E$, as in Lemma\/ {\rm \ref{nt.6}}.

Let $M$ be a non-zero $R$-module of finite length with the property
that its zero submodule is irreducible, that is, cannot be expressed
as the intersection of two non-zero submodules. Then there exists a
homogeneous $R[x,f]$-homomorphism
\[
\lambda := \bigoplus_{i\in \nn}\lambda_i : R[x,f]\otimes_RM =
\bigoplus_{i\in \nn}(Rx^i\otimes_RM) \lra \widetilde{E}
\]
such that $\lambda_0$ is a monomorphism.
\end{lem}

\begin{proof} Since $E_R(M) \cong E_R(R/\fm)$, there exists
an $R$-monomorphism $\lambda_0 : M \lra E$. We can then define, for
each $n \in \N$, an $R$-homomorphism $\lambda_n: Rx^n \otimes_RM
\lra (\widetilde{E})_n = E$ for which $\lambda_n(rx^n \otimes m) =
rx^n\lambda_0(m)$ for all $r \in R$ and all $m \in M$. It is
straightforward to check that the $\lambda_n~(n\in\nn)$ provide a
homogeneous $R[x,f]$-homomorphism as claimed.
\end{proof}

\begin{rmk}\label{ext.22} Let $M$ be an $R$-module and let $h,n \in
\nn$. Endow $Rx^n$ and $Rx^h$ with their natural structures as
$(R,R)$-bimodules (inherited from their being graded components of
$R[x,f]$). Then there is an isomorphism of (left) $R$-modules $\phi
: Rx^{n+h}\otimes_R M \stackrel{\cong}{\lra}
Rx^{n}\otimes_R(Rx^{h}\otimes_R M)$ for which $\phi(rx^{n+h} \otimes
m) = rx^n \otimes (x^h \otimes m)$ for all $r \in R$ and $m \in M$.
\end{rmk}

\begin{rmk}\label{ext.22a} It was pointed out in the Introduction
that $R$ is $F$-pure precisely when, for each $R$-module $N$, the
map $\psi_N : N \lra Rx\otimes_RN$ for which $\psi_N(g) = x\otimes
g$ for all $g \in N$ is injective. In view of isomorphisms like
those described in Remark \ref{ext.22} above, this is the case if
and only if, for each $R$-module $N$, the left $R[x,f]$-module
$R[x,f]\otimes_RN$ is $x$-torsion-free; since tensor product
commutes with direct limits, we can conclude that $R$ is $F$-pure if
and only if, for each finitely generated $R$-module $N$, the left
$R[x,f]$-module $R[x,f]\otimes_RN$ is $x$-torsion-free.
\end{rmk}

\begin{lem}\label{ext.22b} Let $\fd$ be an ideal of $R$ of positive
height.  Then $\fd$ can be generated by the elements in $\fd \cap
R^{\circ}$.
\end{lem}

\begin{proof} Let $\fd'$ be the ideal generated by the elements of $\fd \cap
R^{\circ}$; of course, $\fd' \subseteq \fd$. Let $\fp_1, \ldots,
\fp_t$ be the minimal prime ideals of $R$. Then $\fd \subseteq \fd'
\cup \fp_1 \cup \cdots \cup \fp_t$, and so, by prime avoidance,
either $\fd \subseteq \fp_i$ for some $i \in \{1, \ldots, t\}$, or
$\fd \subseteq \fd'$. Since $\height \fd \geq 1$, we must have $\fd
\subseteq \fd'$.
\end{proof}

\begin{thm}\label{gor.10}
Suppose that $(R,\fm)$ is local, and that there exists a left
$R[x,f]$-module $E$ which, as $R$-module, is isomorphic to $
E_R(R/\fm)$. Let $M$ be a finitely generated $R$-module. Then, for
each $n \in \nn$, there is a countable family
$\left(E_{ni}\right)_{i \in Y_n}$ of $\nn$-graded left
$R[x,f]$-modules, each $x$-torsion-free if $E$ is, and with
$\mathcal{G}(E_{ni}/\Gamma_x(E_{ni})) = \mathcal{G}(E/\Gamma_x(E))$
for all $i \in Y_n$, for which there exists a homogeneous
$R[x,f]$-monomorphism
\[
\nu : R[x,f]\otimes_RM = \bigoplus_{i\in \nn}(Rx^i\otimes_RM) \lra
\prod_{\stackrel{{\scriptstyle n\in\nn}}{i \in Y_n}}{\textstyle
^{^{^{\!\!\Large \prime}}}} E_{ni}.
\]

In particular, if $E$ is $x$-torsion-free, then \begin{enumerate}
\item $R$ is reduced and has a test element (for modules), \item
$R[x,f]\otimes_RM$ is $x$-torsion-free and
$\mathcal{I}(R[x,f]\otimes_RM) \subseteq \mathcal{I}(E)$, \item $R$
is $F$-pure, \item the unique smallest ideal $\fb$ of positive
height in $\mathcal{I}(E)$ is contained in the test ideal $\tau(R)$
of $R$, and \item $\tau(R) \in \mathcal{I}(E)$.\end{enumerate}
\end{thm}

\begin{proof} For each $n \in \nn$, the (left) $R$-module $Rx^n
\otimes_RM$ is finitely generated, and so $\bigcap_{j \in
\N}\fm^j(Rx^n \otimes_RM) = 0$; therefore the zero submodule of
$Rx^n \otimes_RM$ can be expressed as the intersection of a
countable family $(Q_{ni})_{i\in Y_n}$ of irreducible submodules of
finite colength.

Construct the graded left $R[x,f]$-module $\widetilde{E}$ from $E$,
as in Lemma \ref{nt.6}. By Lemma \ref{gor.9}, there is, for each $n
\in \nn$ and $i \in Y_n$, a homogeneous $R[x,f]$-homomorphism
$R[x,f]\otimes_R\left((Rx^n\otimes_R M)/Q_{ni}\right) \lra
\widetilde{E}$ which is monomorphic in degree $0$. If we now compose
this with the natural homogeneous $R[x,f]$-epimorphism
$$R[x,f]\otimes_R(Rx^n\otimes_R M) \lra
R[x,f]\otimes_R\left((Rx^n\otimes_R M)/Q_{ni}\right)$$ and use
isomorphisms of the type described in Remark \ref{ext.22}, we obtain
(after application of the shift functor $ (\: {\scriptscriptstyle
\bullet} \:)(-n)$) a homogeneous $R[x,f]$-homomorphism
$$
\lambda'_{ni} : \bigoplus_{j \geq n}(Rx^j\otimes_RM) \lra
\widetilde{E}(-n)
$$
for which the $n$th component has kernel equal to $Q_{ni}$.

We can now use Corollary \ref{ext.23} to extend $\lambda'_{ni}$ by
$n$ places to produce a homogeneous $R[x,f]$-homomorphism
$$
\lambda_{ni} : \bigoplus_{j \geq 0}(Rx^j\otimes_RM) =
R[x,f]\otimes_RM \lra E_{ni},
$$
where $E_{ni}$ is an appropriate $n$-place extension of
$\widetilde{E}(-n)$, for which the $n$th component has kernel equal
to $Q_{ni}$. Note that, by Lemma \ref{nt.6a}, Remark \ref{nt.8a} and
Definition \ref{nt.9}, we have $\mathcal{G}(E_{ni}/\Gamma_x(E_{ni}))
= \mathcal{G}(E/\Gamma_x(E))$ and that $E_{ni}$ is $x$-torsion-free
if $E$ is.

There is therefore a homogeneous $R[x,f]$-homomorphism
\[
\nu  = \bigoplus_{j\in \nn}\nu_j : R[x,f]\otimes_RM \lra
\prod_{\stackrel{{\scriptstyle n\in\nn}}{i \in Y_n}}{\textstyle
^{^{^{\!\!\Large \prime}}}} E_{ni} =: K
\]
such that $\nu_j(\xi) =
\big((\lambda_{ni})_j(\xi)\big)_{n\in\nn,i\in Y_n}$ for all $j \in
\nn$ and $\xi \in Rx^j \otimes_RM$. For each $j \in \nn$, the zero
submodule of $Rx^j\otimes_RM$ is equal to $\bigcap_{i\in
Y_j}Q_{ji}$, and this means that $\nu_j$ is a monomorphism. Hence
$\nu$ is an $R[x,f]$-monomorphism.

Now suppose for the remainder of the proof that $E$ is
$x$-torsion-free. By Lemmas \ref{nt.6a} and \ref{nt.8}, we see that
$E_{ni}$ is $x$-torsion-free, for all $n \in \nn$ and all $i \in
Y_n$. We now deduce from Lemma \ref{nt.5}(iii) that $K$ is
$x$-torsion-free, so that $R[x,f]\otimes_RM$ is $x$-torsion-free in
view of the $R[x,f]$-monomorphism $\nu$. As this is true for each
choice of finitely generated $R$-module $M$, it follows from Remark
\ref{ext.22a} that $R$ is $F$-pure, and therefore reduced.

The $R[x,f]$-monomorphism $\nu$ also shows that
$\mathcal{G}(R[x,f]\otimes_RM) \subseteq \mathcal{G}(K)$, while
Lemmas \ref{nt.5}(ii), \ref{nt.6a} and \ref{nt.8} show that
$\mathcal{G}(K) = \mathcal{G}(\widetilde{E}) = \mathcal{G}(E)$.
Therefore $$\mathcal{I}(R[x,f]\otimes_RM) \subseteq
\mathcal{I}(E).$$ Since $E_R(R/\fm)$ is Artinian as an $R$-module,
it follows from \cite[Corollary 3.11 and Theorem 3.12]{ga} that
there exists a unique smallest ideal $\fb$ of positive height in
$\mathcal{I}(E)$, and that any element of $R[x,f]\otimes_RM$ that is
annihilated by $\bigoplus_{n \geq n_0}Rcx^n$ for some $c \in
R^{\circ}$ and $n_0 \in \nn$ must also be annihilated by $\fb
R[x,f]$. This means that each element of $\fb \cap R^{\circ}$ (note
that this set generates $\fb$, by Lemma \ref{ext.22b}) is a test
element (for modules) for $R$.

Next, note that $\ann_{R[x,f]\otimes_RM}\tau(R)R[x,f] =
\bigoplus_{n\in\nn} 0^*_{Rx^n\otimes_RM}$, and so, by the
immediately preceding paragraph, the graded annihilator of this will
be $\fb_MR[x,f]$ for some $\fb_M \in \mathcal{I}(E)$ for which
$\tau(R) \subseteq \fb_M$.

By \cite[Corollary 3.7]{ga}, each member of $\mathcal{I}(E)$ is the
intersection of the members of a subset of the finite set
$\mathcal{I}(E) \cap \Spec(R)$. It therefore follows that $\tau(R)$
is the intersection of all members of the finite set $\{ \fp \in
\mathcal{I}(E) \cap \Spec(R) : \height \fp \geq 1\}$, and that
$\fb_M$ is the intersection of the members of some subset of this
set.

Set $\fd := \bigcap_{M}\fb_M$, where the intersection is taken over
all finitely generated $R$-modules $M$. It follows from the above
paragraph that $\fd$ is the intersection of the members of some
subset of $\{ \fp \in \mathcal{I}(E) \cap \Spec(R) : \height \fp
\geq 1\}$, so that it belongs to $\mathcal{I}(E)$ by \cite[Corollary
1.12]{ga}.  Note that $\height \fd \geq 1$.

We suppose that $\tau(R) \subset \fd$ and seek a contradiction. (The
symbol `$\subset$' is reserved to denote strict inclusion.) Now
$\fd$ can, by Lemma \ref{ext.22b}, be generated by elements in $\fd
\cap R^{\circ}$, and so there exists $a \in \fd \cap R^{\circ}
\setminus \tau(R)$. Then $a$ is not a test element for $R$, and yet
it belongs to $R^{\circ}$. This means that there must exist a
finitely generated $R$-module $N$ and an element $y \in 0^*_N$ such
that $1\otimes y \in (R[x,f]\otimes_RN)_0$ is not annihilated by
$(aR)R[x,f]$. Therefore $a \not\in \fb_N$, and this is a
contradiction. Therefore $\tau(R) = \fd \in \mathcal{I}(E)$.
\end{proof}

\begin{rmk}\label{gor.10aa} In the special case of Theorem
\ref{gor.10} in which $E$ is $x$-torsion-free, the main thrust of
the argument presented in the above proof comes from the use of the
homogeneous $R[x,f]$-monomorphism $\nu$, in conjunction with the
unique smallest ideal $\fb$ of positive height in $\mathcal{I}(E)$,
to establish the existence of a test element for $R$. The conclusion
(in part (iii)) that $R$ is $F$-pure comes as a byproduct. I am
grateful to the referee for pointing out the following short
alternative proof of the fact that, in these circumstances, $R$ is
$F$-pure.

Let $\kappa : E \lra Rx \otimes_RE$ be the Abelian group
homomorphism for which $\kappa(e) = x \otimes e$ for all $e \in E$,
and let $\mu : Rx \otimes_RE \lra E$ be the $R$-homomorphism for
which $\mu(rx \otimes e) = rxe$ for all $r \in R$ and $e \in E$.
Then the map $\mu\kappa:E \lra E$ satisfies $\mu\kappa(e) = xe$ for
all $e \in E$, and so is injective because $E$ is $x$-torsion-free.
Hence $\kappa$ is injective. Therefore $R$ is $F$-pure by a result
of M. Hochster and J. L. Roberts \cite[Proposition 6.11]{HocRob74}.
\end{rmk}

Before we deduce the main result of the paper from Theorem
\ref{gor.10}, we draw attention to other known results that relate
the test ideal of the local ring $(R,\fm)$ to annihilators of
submodules of the $R$-module $E(R/\fm)$.

\begin{disc}\label{gor.10a} Recall from Hochster--Huneke \cite[Definition
(8.22)]{HocHun90} that, in general, even in the case where $R$ does
not have a test element, the test ideal $\tau(R)$ of $R$ is defined
to be $\bigcap_{M}(0 :_R0^*_M)$, where the intersection is taken
over all finitely generated $R$-modules $M$. If $R$ has a test
element (for modules), then $\tau(R)$ is the ideal of $R$ generated
by all such test elements, and $\tau(R)\cap R^{\circ}$ is the set of
all test elements (for modules) for $R$. (See \cite[Proposition
(8.23)(b)]{HocHun90}.)

Henceforth in this discussion, we assume that $(R,\fm)$ is local,
and we set $E := E_R(R/\fm)$.

\begin{enumerate}
\item Recall from Hochster--Huneke \cite[Definition (8.19)]{HocHun90} that
the {\em finitistic tight closure of $0$ in $E$,\/} denoted by
$0^{*{\text f}{\text g}}_E$, is defined to be $\bigcup_{M}0^*_M$,
where the union is taken over all finitely generated $R$-submodules
$M$ of $E$. It was shown in \cite[Proposition (8.23)(d)]{HocHun90}
that $\tau (R) = (0:_R0^{*{\text f}{\text g}}_E)$.
\item In the case where $R$ is a reduced ring that is a homomorphic
image of an excellent regular local ring of characteristic $p$,
results of G. Lyubeznik and K. E. Smith in \cite[\S 7]{LyuSmi01}
establish some very good properties of the ideal
$$\widetilde{\tau}(R) = (0:_R0^{*}_E).$$
Some of their results were extended by I. M. Aberbach and F. Enescu,
in \cite[Theorem 3.6]{AbeEnu02}, to the case where $R$ is a reduced
excellent local ring.

Notice that, by part (i), the ideal $\widetilde{\tau}(R)$ is equal
to the test ideal of $R$ in the case where $0^{*}_E = 0^{*{\text
f}{\text g}}_E$.
\item In the case where $R$ is Gorenstein, or merely
quasi-Gorenstein, we have $E \cong H^{\dim R}_{\fm}(R)$; then,
provided $R$ is excellent and equidimensional (Gorenstein local
rings are automatically equidimensional), it follows from work of K.
E. Smith \cite[Proposition 3.3]{kes2} that $0^{*}_E = 0^{*{\text
f}{\text g}}_E$, so that, by part (ii), $$\tau(R) = (0:_R0^{*}_E).$$
\item Recall from the Introduction that membership of $0^{*}_E$ is
(essentially) defined in terms of the natural (graded) left
$R[x,f]$-module structure on $R[x,f] \otimes_R E =
\bigoplus_{n\in\nn} (Rx^n \otimes_R E)$: we have that $m \in E$
belongs to $0^{*}_E$ if and only if there exists $c \in R^{\circ}$
and $n_0 \in \nn$ such that the element $$1 \otimes m \in
(R[x,f]\otimes_RE)_0$$ is annihilated by $\bigoplus_{j\geq
n_0}Rcx^j$. In the case where $R$ is Gorenstein, or merely
quasi-Gorenstein, we have $E \cong H^{\dim R}_{\fm}(R) =: H$, and
the latter $R$-module carries a natural structure as left
$R[x,f]$-module (as recalled in \cite[Reminder 4.1]{ga}). If we
construct the graded left $R[x,f]$-module $\widetilde{H}$ from $H$
as in Lemma \ref{nt.6}, then it follows from \cite[Remark
4.2(iii)]{ga} that there are $R[x,f]$-isomorphisms
$$
R[x,f] \otimes_R E \cong R[x,f] \otimes_R H \cong \widetilde{H}.
$$
\item Thus, in the special case in which $R$ is Gorenstein, the conclusions
of Theorem \ref{gor.10}(i)--(v) will not surprise experts in tight
closure theory; however, in more general situations, the approach
taken in that theorem presents a new perspective on the test ideal.
\end{enumerate}
\end{disc}

\begin{cor}\label{gor.11} Suppose that $(R,\fm)$ is local and complete, and that there
exists an $x$-torsion-free left $R[x,f]$-module $E$ which, as
$R$-module, is isomorphic to $E_R(R/\fm)$. Then, for each ideal $\fc
\in \mathcal{I}(E)$, the following hold:

\begin{enumerate}
\item there is an isomorphism of $R/\fc$-modules $\ann_E(\fc) \cong
E_{R/\fc}\left((R/\fc)/(\fm/\fc)\right)$, and $\ann_E(\fc)$ inherits
from $E$ a structure as $x$-torsion-free left $(R/\fc)[x,f]$-module
for which
$$
\mathcal{I}_{R/\fc}(\ann_E(\fc)) = \left\{ \fd/\fc: \fd \in
\mathcal{I}(E) \text{~and~} \fd \supseteq \fc \right\}\mbox{;}
$$
\item the complete local ring $\overline{R} := R/\fc$ is $F$-pure and
reduced, and has a test element (for modules); and
\item the test ideal of $\overline{R}$ is $\fd/\fc$ for some ideal $\fd \in
\mathcal{I}(E)$ with $\height (\fd/\fc) \geq 1$.
\end{enumerate}

We conclude that there is a strictly ascending chain
$$
0 = \tau_0 \subset \tau_1 \subset \cdots \subset \tau_{n-1} \subset
\tau_n = R
$$
of ideals of $R$, all belonging to $\mathcal{I}(E)$, such that, for
all $i = 0, \ldots, n-1$, the ring $R/\tau_i$ is $F$-pure and
reduced, and has a test element (for modules), and its test ideal is
$\tau_{i+1}/\tau_i$.

\end{cor}

\begin{notes} The referee has pointed out that, in the special case
of Corollary \ref{gor.11} in which $R$ is (also) Gorenstein and
$F$-finite, the conclusion that $R/\fc$ is $F$-pure in part (ii)
follows from Enescu and Hochster \cite[Discussion 2.6 and Theorem
4.1]{EneHocarX} (and Proposition \ref{gor.3} above).

Note that the final conclusion of Corollary \ref{gor.11} is similar
to J. Cowden Vassilev's result in \cite[\S 4]{Cowde98} that, if $R$
is a (not necessarily complete) $F$-pure homomorphic image of an
$F$-finite regular local ring, then there exists a strictly
ascending chain $$0 = \tau_0 \subset \tau_1 \subset \cdots \subset
\tau_t = R$$ of radical ideals of $R$ such that, for each $i = 0,
\ldots, t-1$, the reduced local ring $R/\tau_i$ is $F$-pure and its
test ideal (has positive height and) is exactly $\tau_{i+1}/\tau_i$.
\end{notes}

\begin{proof} By Proposition \ref{gor.3}, we have
$\ann_E\left(\bigoplus_{n\in\nn}\fc x^n\right) = \ann_E(\fc)$, and
this is an $R[x,f]$-submodule of $E$ that is annihilated by $\fc$.
Use overlines to denote natural images in $R/\fc$ of elements of
$R$. It is easy to use \cite[Lemma 1.3]{KS} to see that
$\ann_E(\fc)$ inherits a structure as left $(R/\fc)[x,f]$-module
with $\overline{r}e = re$ for all $r \in R$ and $e \in \ann_E(\fc)$
(and multiplication by $x$ on an element of $\ann_E(\fc)$ as in the
$R[x,f]$-module $E$). Note that $\ann_E(\fc)$ is $x$-torsion-free,
and that
$$
\mathcal{I}_{R/\fc}(\ann_E(\fc)) = \left\{ \fd/\fc: \fd \in
\mathcal{I}(E) \text{~and~} \fd \supseteq \fc \right\}.
$$

Now, as $R/\fc$-module, $\ann_E(\fc) \cong
E_{R/\fc}\left((R/\fc)/(\fm/\fc)\right)$ (by \cite[Lemma
10.1.15]{LC}, for example). Thus part (i) is proved. We can also
apply Theorem \ref{gor.10} to the left $(R/\fc)[x,f]$-module
$\ann_E(\fc)$ to prove parts (ii) and (iii); the same theorem shows
that the test ideal $\tau(R/\fc)$ of $R/\fc$ belongs to
$\mathcal{I}_{R/\fc}(\ann_E(\fc))$. Note also that $R/\fc$ satisfies
the hypotheses of the corollary, so that conclusions (i), (ii) and
(iii) are valid for $R/\fc$ and $\ann_E(\fc)$; with these
observations, it is easy to prove the final claim by induction.
\end{proof}

We end the paper with some comments about sources of examples of
local rings that satisfy the hypotheses of Corollary \ref{gor.11}.

\begin{ex}\label{gor.12} Suppose that $(R,\fm)$ is local, complete, $F$-injective and
quasi-Goren\-stein. (Note that Gorenstein local rings are
quasi-Gorenstein, and that the properties of being $F$-injective and
quasi-Gorenstein are inherited by the completion of a non-complete
local ring of characteristic $p$.)

Let $d := \dim R$ and $H := H^d_{\fm}(R)$, and note that $H$ has a
natural left $R[x,f]$-module structure (recalled in \cite[Reminder
4.1]{ga}); since $R$ is $F$-injective, $H$ is $x$-torsion-free; and
since $R$ is quasi-Gorenstein, there is an $R$-isomorphism $H \cong
E_R(R/\fm)$.

It therefore follows from Corollary \ref{gor.11} that, for each
ideal $\fb \in \mathcal{I}(H)$, the complete reduced local ring
$R/\fb$ satisfies the hypotheses of that corollary. This, together
with Fedder's criterion for $F$-purity \cite[Theorem 1.12]{Fedde83}
ensures a good supply of examples of rings to which the results of
this paper apply.
\end{ex}

\begin{ex}\label{ex.13} Suppose that $(R,\fm)$ is local and complete, and that there
exists an $x$-torsion-free left $R[x,f]$-module $E$ which, as
$R$-module, is isomorphic to $E_R(R/\fm)$. Note that these
hypotheses imply that $R$ is reduced, by Corollary \ref{gor.11}. Let
$\fp $ be a minimal prime ideal of $R$.

Observe that $(0:_R E) = 0$, and so $0 \in \mathcal{I}(E)$.
Therefore, by \cite[Theorem 3.6]{ga}, we have $\fp \in
\mathcal{I}(E)$, and it follows from Corollary \ref{gor.11} that the
complete local domain $R/\fp$ satisfies the hypotheses of that
corollary.
\end{ex}

Another example is provided by M. Katzman's work in \cite{MK}.

\begin{ex}\label{mksect9} Let $\F_2$ be the field of two elements,
let $T_1, T_2, T_3, T_4, T_5$ be independent indeterminates, and let
$R:= \F_2[[T_1, T_2, T_3, T_4, T_5]]/\fd$, where $\fd$ is the ideal
of $\F_2[[T_1, T_2, T_3, T_4, T_5]]$ generated by the $2 \times 2$
minors of the matrix
$$
\left(\begin{array}{cccc} T_1 & T_2 & T_2 & T_5\\T_4 & T_4 & T_3 &
T_1 \end{array} \right).
$$
Then $R$ is Cohen--Macaulay but not Gorenstein, and it follows from
the calculations reported by Katzman in \cite[\S 9]{MK} that the
injective envelope $E$ of the simple $R$-module carries a structure
of $x$-torsion-free left $R[x,f]$-module. Thus the conclusions of
Corollary \ref{gor.11} apply to this $R$.
\end{ex}

\end{document}